\documentclass[reqno,a4paper,draft]{amsart}
\usepackage{enumitem}

\setenumerate{label=\textnormal{(\arabic*)}}

\usepackage{amsmath,amssymb,dsfont,verbatim,mathtools,bm,geometry,fge,endnotes}

\usepackage{booktabs}

\usepackage[latin1]{inputenc}
\usepackage[raggedright]{titlesec}
\usepackage{mathtools}
\usepackage{tikz}
\usepackage{float,subfig}
\usepackage{amsrefs}

\titleformat{\chapter}[display]
{\normalfont\huge\bfseries}{\chaptertitlename\\thechapter}{20pt}{\Huge}
\titleformat{\section}
{\normalfont\Large\bfseries\center}{\thesection}{1em}{}
\titleformat{\subsection}
{\normalfont\large\bfseries}{\thesubsection}{1em}{}
\titleformat{\subsubsection}[runin]
{\normalfont\normalsize\bfseries}{\thesubsubsection}{1em}{}
\titleformat{\paragraph}[runin]
{\normalfont\normalsize\bfseries}{\theparagraph}{1em}{}
\titleformat{\subparagraph}[runin]
{\normalfont\normalsize\bfseries}{\thesubparagraph}{1em}{}
\titlespacing*{\chapter} {0pt}{50pt}{40pt}
\titlespacing*{\section} {0pt}{3.5ex plus 1ex minus .2ex}{2.3ex plus .2ex}
\titlespacing*{\subsection} {0pt}{3.25ex plus 1ex minus .2ex}{1.5ex plus .2ex}
\titlespacing*{\subsubsection}{0pt}{3.25ex plus 1ex minus .2ex}{1.5ex plus .2ex}
\titlespacing*{\paragraph} {0pt}{3.25ex plus 1ex minus .2ex}{1em}
\titlespacing*{\subparagraph} {\parindent}{3.25ex plus 1ex minus .2ex}{1em}


\newtheorem{theorem}{Theorem}[section]
\newtheorem{lemma}[theorem]{Lemma}
\newtheorem{proposition}[theorem]{Proposition}

\theoremstyle{definition}

\newtheorem{notation}[theorem]{Notation}

\theoremstyle{remark}
\newtheorem{remark}[theorem]{Remark}

\DeclareMathOperator{\Supp}{Supp}

\DeclareMathOperator{\en}{en}

\DeclareMathOperator{\st}{st}

\DeclareMathOperator{\Succ}{Succ}

\DeclareMathOperator{\Dir}{Dir}

\newcommand{\ov}{\overline}

\begin{document}

\title{A short and elementary proof of Jung's theorem}

\author{Jorge A. Guccione}
\address{Departamento de Matem\'atica\\ Facultad de Ciencias Exactas y Naturales-UBA,
Pabell\'on~1-Ciudad Universitaria\\ Intendente Guiraldes 2160 (C1428EGA) Buenos Aires, Argentina.}
\address{Instituto de Investigaciones Matem\'aticas ``Luis A. Santal\'o"\\ Facultad de
Ciencias Exactas y Natu\-ra\-les-UBA, Pabell\'on~1-Ciudad Universitaria\\ Intendente
Guiraldes 2160 (C1428EGA) Buenos Aires, Argentina.}
\email{vander@dm.uba.ar}
\thanks{Supported by  UBACYT 095, PIP 112-200801-00900 (CONICET) and PUCP-DGI-2013-3036}

\author{Juan J. Guccione}
\address{Departamento de Matem\'atica\\ Facultad de Ciencias Exactas y Naturales-UBA\\
Pabell\'on~1-Ciudad Universitaria\\ Intendente Guiraldes 2160 (C1428EGA) Buenos Aires, Argentina.}
\address{Instituto Argentino de Matem\'atica-CONICET\\ Savedra 15 3er piso\\ (C1083ACA)
Buenos Aires, Argentina.}
\email{jjgucci@dm.uba.ar}
\thanks{Supported by  UBACYT 095 and PIP 112-200801-00900 (CONICET)}

\author[C. Valqui]{Christian Valqui}
\address{Pontificia Universidad Cat\'olica del Per\'u, Secci\'on Matem\'aticas, PUCP,
Av. Universitaria 1801, San Miguel, Lima 32, Per\'u.}
\address{Instituto de Matem\'atica y Ciencias Afines (IMCA) Calle Los Bi\'ologos 245.
Urb San C\'esar. La Molina, Lima 12, Per\'u.}
\email{cvalqui@pucp.edu.pe}
\thanks{Christian Valqui was supported by PUCP-DGI-2013-3036}


\begin{abstract} We give a short and elementary proof of Jung's theorem, which states that for
a field $K$ of characteristic zero the
automorphisms of $K[x,y]$ are generated by elementary automorphisms and linear automorphisms.
\end{abstract}

\maketitle

\section*{Introduction} The theorem of Jung~\cite{Ju} states that if $K$  is a field of characteristic
zero, then any automorphism of $L:=K[x,y]$ is the finite composition of elementary automorphisms
(given by $x\mapsto x,\ y\mapsto y+p(x)$ or by $y\mapsto y,\ x\mapsto x+p(y)$) and linear
automorphisms. Many authors have given proofs of this fact, for example in~\cite{AM},
 \cite{Gu}, \cite{MW}, \cite{N}, \cite{Ng}, \cite{R} and~\cite{vdK}, the last and very short and
elegant proof in~\cite{ML}, which works in every algebraically closed field of characteristic zero.
The key step in the proof of~\cite{ML} is the same as in ours:
In the situation of the figure,
$$
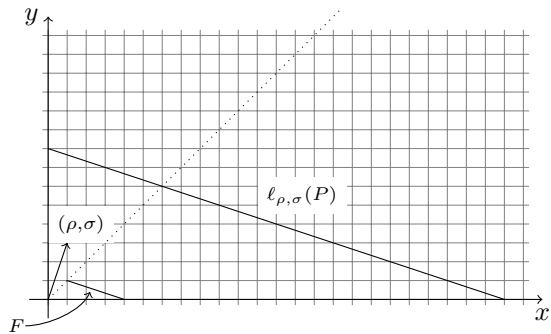
\begin{tikzpicture}[scale=0.25]
\draw[step=1cm,gray,very thin] (-0.3,-0.3) grid (25.3,14.3);
\draw (0,4) node[fill=white, above=0pt, right=0pt]{$\scriptstyle (\rho, \sigma)$};
\draw (11,5.5) node[fill=white, above=0pt, right=0pt]{$\scriptstyle \ell_{\rho,\sigma}(P)$};
\draw[dotted] (0,0) --(15.5,15.5);
\draw [->] (-1,0)--(26,0) node[anchor=north]{$x$};
\draw [->] (0,-1)--(0,15) node[anchor=east]{$y$};
\draw [->] (0,0)--(1,3);
\draw [-] (0,8)--(24,0);
\draw [-] (1,1)--(4,0);
\draw (-2.6,-1.4) node[ above=0pt, right=0pt]{$\scriptstyle F$};
\draw[->] (-1.2, -1.4) ..controls (0,-1.4) and (1.8,-0.6) .. (2.2,0.4);
\end{tikzpicture}
$$
there exists a polynomial $F$ (called $\zeta$ in~\cite{ML}) such that $F=\mu x(y+\lambda x^{\sigma})$ and
$[F,\ell_{\rho,\sigma}(P)]=\ell_{\rho,\sigma}(P)$. Then we apply $\varphi$ given by
$\varphi(x):=x$ and $\varphi(y):=y-\lambda x^{\sigma}$, and obtain $\deg(\varphi(P))<\deg(P)$.

Here $[P,Q]$ stands for the
determinant of the jacobian matrix of two polynomials $P,Q$ and $\ell_{\rho,\sigma}(P)$ is the leading form of
$P$ with respect to the weight $(\rho,\sigma)$.  To our knowledge our proof is the shortest and simplest,
(except the proof of~\cite{ML}), and Theorem~\ref{central} is the only fact that we use that is not
straightforward or elementary. The element $F$ can be traced back to 1975 in~\cite{J}. In order to
obtain a proof for a field that is not necessarily algebraically closed, we have to prove that for a
polynomial automorphism there can be only one point at infinity, which we do in
Proposition~\ref{un factor en el infinito}.

\section{Preliminaries}
We first gather notations and results of~\cite{GGV}.  We define the set of directions by
\begin{equation*}
\mathfrak{V} := \{(\rho,\sigma)\in \mathds{Z}^2: \text{$\gcd(\rho,\sigma) = 1$} \}.
\end{equation*}
For all $(\rho,\sigma)\in \mathfrak{V}$ and $(i,j)\in \mathds{Z}\times
\mathds{Z}$ we write $v_{\rho,\sigma}(i,j):= \rho i+\sigma j$ and for
$P = \sum a_{i,j} x^{i} y^j\in L\setminus\{0\}$, we define
\begin{itemize}

\smallskip

\item[-] The {\em support} of $P$ as $\Supp(P) := \left\{\left(i,j\right): a_{i,j}\ne
    0\right\}$.

\smallskip

\item[-] The {\em $(\rho,\sigma)$-degree} of $P$ as $v_{\rho,\sigma}(P):= \max\left\{ v_{\rho,\sigma}
    (i,j): a_{i,j} \ne 0\right\}$.

\smallskip

\item[-] The {\em $(\rho,\sigma)$-leading term} of $P$ as $\ell_{\rho,\sigma}(P):= \displaystyle
    \sum_{\{\rho i + \sigma j = v_{\rho,\sigma}(P)\}} a_{i,j} x^{i} y^j$.

\smallskip

\end{itemize}

 We say that $P\in L$ is {\em
$(\rho,\sigma)$-homogeneous} if $P = \ell_{\rho,\sigma}(P)$.
 We assign to each direction its corresponding unit vector in $S^1$, and we define an
{\em interval} in $\mathfrak{V}$ as the preimage under this map of an arc of $S^1$ that is not
the whole circle. We consider each interval endowed with the order that increases counterclockwise.

For each $P\in L\setminus \{0\}$, we let $H(P)$ denote the
 {\em Newton polygon of $P$}, and it is evident that each one
of its edges is the convex hull of the support
of $\ell_{\rho,\sigma} (P)$, where $(\rho,\sigma)$ is orthogonal to the given edge and points
outside of $H(P)$. This directions form the set
$$
\Dir(P):=\{(\rho,\sigma)\in\mathfrak{V}:\#\Supp(\ell_{\rho,\sigma}(P))>1\}.
$$

\begin{notation}\label{Comienzo y Fin de un elemento de W} Let $(\rho,\sigma)\in \mathfrak{V}$ arbitrary. We
let $\st_{\rho,\sigma}(P)$ and $\en_{\rho,\sigma}(P)$ denote the first and the last point that we find on
$H(\ell_{\rho,\sigma}(P))$ when we run counterclockwise along the boundary of $H(P)$. Note that these points
coincide when $\ell_{\rho,\sigma}(P)$ is a monomial.
\end{notation}

We say that two vectors $A,B\in\mathds{R}^2$ are
{\em aligned} and write $A\sim B$, if 
$0=A\times B:=\det\left(\begin{smallmatrix} a_1 & a_2\\ b_1 & b_2 \end{smallmatrix}\right)$

\begin{proposition}[\cite{GGV}*{Proposition 2.4}]\label{extremosnoalineados} Let $P,Q,R\in L\setminus \{0\}$ be such that
$$
[\ell_{\rho,\sigma}(P),\ell_{\rho,\sigma}(Q)]=\ell_{\rho,\sigma}(R),
$$
where $(\rho,\sigma)\in \mathfrak{V}$. We have:

\begin{enumerate}

\smallskip

\item $\st_{\rho,\sigma}(P)\nsim\st_{\rho,\sigma}(Q)$ if and only if $\st_{\rho,\sigma}(P)
    +\st_{\rho,\sigma}(Q) - (1,1) =\st_{\rho,\sigma}(R)$.

\smallskip

\item $\en_{\rho,\sigma}(P)\nsim\en_{\rho,\sigma}(Q)$ if and only if
    $\en_{\rho,\sigma}(P)+\en_{\rho,\sigma}(Q)-(1,1)=\en_{\rho,\sigma}(R)$.

\end{enumerate}

\end{proposition}

\begin{remark}\label{F no es monomio} Let $(\rho,\sigma)\in \mathfrak{V}$ and let
$P,F\in L$ be $(\rho,\sigma)$-homogeneous such that $[F,P] = P$. If $F$ is a
monomial, then $F = \lambda xy$ with $\lambda\in K^{\times}$, and, either $\rho+\sigma = 0$
or $P$ is also a monomial.
\end{remark}
The following theorem is an important tool in the constructions of~\cite{GGV}. It is the only result
we use that is not straightforward.
\begin{theorem}[\cite{GGV}*{Theorem 2.6}]\label{central} Let $P\in L$ and let $(\rho,\sigma)\in \mathfrak{V}$ be such that
$\rho+\sigma>0$ and $v_{\rho,\sigma}(P)>0$. If $[P,Q]\in K^{\times}$
for some $Q\in L$, then there exists  a $(\rho,\sigma)$-ho\-mo\-ge\-neous element $F\in L$ such that
\begin{equation}\label{eq central}
v_{\rho,\sigma}(F)=\rho+\sigma\quad \text{and}\quad [F,\ell_{\rho,\sigma}(P)]= \ell_{\rho,\sigma}(P).
\end{equation}
\end{theorem}
 If $I$
is an interval in $\mathfrak{V}$ and if there is no closed half circle contained in $I$,  then for  all $(\rho,\sigma),(\rho_1\sigma_1)\in I$
we have
\begin{equation}\label{order}
(\rho_1,\sigma_1)<(\rho,\sigma) \Longleftrightarrow (\rho_1,\sigma_1)\times (\rho,\sigma)>0.
\end{equation}

\begin{proposition}[\cite{GGV}*{Proposition 3.7}]\label{le basico} Let $P\!\in\! L\!\setminus\!\{0\}$ and let
$(\rho_1,\sigma_1)$ and $(\rho_2,\sigma_2)$ be consecutive elements in $\Dir(P)$. If $(\rho_1,\sigma_1)<
(\rho,\sigma) <(\rho_2,\sigma_2)$, then \ $\en_{\rho_1,\sigma_1}(P) = \Supp(\ell_{\rho,\sigma}(P)) =
\st_{\rho_2,\sigma_2}(P)$.
\end{proposition}

\begin{proposition}[\cite{GGV}*{Proposition 3.10}]\label{varphi preserva el Jacobiano} Let $P,Q\in L$ and
$\varphi\colon L\to L$ an algebra morphism. Then
\begin{equation}\label{Chain rule for Jacobian}
[\varphi(P),\varphi(Q)]=\varphi([P,Q])[\varphi(x),\varphi(y)].
\end{equation}
\end{proposition}
\bigskip
\section{Jung's Theorem}
\begin{lemma}\label{esquina en el eje}
Let $f:K[x,y]\to K[x,y]$ be an automorphism, set $P:=f(x)$ and
assume $(a,b)\in\{\en_{1,1}(P),\st_{1,1}(P)\}$. Then $a=0$ or $b=0$.
\end{lemma}
\begin{proof}
We will proof only the case $(a,b)=\st_{1,1}(P)$, since the argument in the other case is the same.
Assume $a\ge b>0$. We set $R_0:=x$ and $R_{j+1}:=[R_j,P]$. Then
$\st_{1,1}(R_0)=(1,0)\nsim (a,b)=\st_{1,1}(P)$ and so, by Proposition~\ref{extremosnoalineados}(1)
$$
\st_{1,1}(R_1)=(1,0)+(a,b)-(1,1)\nsim (a,b).
$$
Increasing $k$ and using Proposition~\ref{extremosnoalineados}(1), one obtains inductively that
$$
\st_{1,1}(R_k)=(1,0)+k(a,b)-k(1,1)\nsim (a,b),
$$
since $((1,0)+k(a,b)-k(1,1))\times (a,b)=b+k(a-b)\ne 0$. Consequently
$R_k\ne0$ for all $k$, which contradicts the fact that $R_n=0$ for $n>>0$.

If $b\ge a>0$, then we set $R_0:=y$ and $R_{j+1}:=[R_j,P]$ and the same argument yields a contradiction.
Hence $a=0$ or $b=0$, as desired.
\end{proof}
The next proposition shows that for an automorphism $f$, there can be only one factor at infinity, 
or equivalently,
that $\ell_{1,1}(f(x))$ is the power of one linear factor.
\begin{proposition}\label{un factor en el infinito}
Let $f:K[x,y]\to K[x,y]$ be an automorphism and set $P:=f(x)$. Then
$$
\Supp(\ell_{1,1}(P)=\{(a,0)\},\quad \Supp(\ell_{1,1}(P)=\{(0,a)\}\quad \text{or}\quad
\ell_{1,1}(P)=\mu(x-\lambda y)^a,
$$
for some  $\mu,\lambda\in K^{\times}$, where $a=v_{1,1}(P)$.
\end{proposition}
\begin{proof} We can assume that $K$ is algebraically closed.
Let $a=v_{1,1}(P)>0$ and write $\ell_{1,1}(P)=x^a p(z)$, where $z:=x^{-1}y$ and $p(z)\in K[z]$.
If $0<b:=\deg(p(z))<a$, then $\en_{1,1}(P)=(a,0)+b(-1,1)=(a-b,b)$ which contradicts
Lemma~\ref{esquina en el eje}. If $\deg(p)=0$, then $\Supp(\ell_{1,1}(P)=\{(a,0)\}$, so it
suffices to consider the case $\deg(p(z))=a$. If neither
$$
\Supp(\ell_{1,1}(P)=\{(0,a)\}\quad \text{nor}\quad
\ell_{1,1}(P)=\mu(x-\lambda y)^a,
$$
then $p(z)=\mu \displaystyle\prod_{i=1}^{k}(z-\lambda_i)^{m_i}$ has a root $\lambda_{i_0}$
with multiplicity $0<m_{i_0}<a$. But then the automorphism $\varphi$ given by $\varphi(x):=x$ and
$\varphi(y):=y+\lambda_{i_0}x$ yields
$$
\ell_{1,1}(\varphi(P))=\varphi(\ell_{1,1}(P))= x^a p(z+\lambda_{i_0})=
\mu x^az^{m_{i_0}}\prod_{{i=1}\atop i\ne i_0}^{k}(z-\ov{\lambda}_i)^{m_i},
$$
where $\ov{\lambda}_i=\lambda_i-\lambda_{i_0}$ and
$\displaystyle\prod_{{i=1}\atop i\ne i_0}^{k}\ov{\lambda}_i^{m_i}\ne 0$.
This implies
$$
\st_{1,1}(\varphi(P))=(a,0)+m_{i_0}(-1,1)=(a-m_{i_0},m_{i_0}),
$$
where $a-m_{i_0}\ne 0$ and $m_{i_0}\ne 0$, which contradicts Lemma~\ref{esquina en el eje}
and concludes the proof.
\end{proof}

\begin{theorem}
Let $f:K[x,y]\to K[x,y]$ be an automorphism. Then
$f$ is the composition of elementary automorphisms and linear automorphisms.
\end{theorem}
\begin{proof} Set $P:=f(x)$. If $\deg(P)=1$, then we can assume $P=x$, and then
$f(y)=\lambda y+q(x)$; it follows that $f$ is the composition of elementary automorphisms and linear
automorphisms.

Therefore it suffices to  prove that either $\deg(P)=v_{1,1}(P)=1$ or there exists an elementary
automorphism $\varphi$ such that $\deg (\varphi(P))<\deg(P)$. By Proposition~\ref{un factor en el infinito}
we can assume
$$
\Supp(\ell_{1,1}(P))=\{(a,0)\}
$$
for some $a\in\mathds{N}$. In fact, if $\Supp(\ell_{1,1}(P)=\{(0,a)\}$
we apply the flip, which is a composition of elementary automorphisms, and if
$\ell_{1,1}(P)=\mu(x-\lambda y)^a$, then we apply the elementary automorphism
given by $x\mapsto x+\lambda y$ and $y\mapsto y$.

We also can assume that $P$ is not a monomial and so we set
$(\rho,\sigma):=\Succ_{P}(1,1)$, the {\em successor} of $(1,1)$, which
is the first element of $\Dir(P)$ that one encounters starting from $(1,1)$ and
running counterclockwise.

If $(\rho,\sigma)\ge (0,1)$, then from
Proposition~\ref{le basico} we obtain $(a,0)=\st_{0,1}(P)$ and then for all $(i,j)\in \Supp(P)$
we have $j=v_{0,1}(i,j)\le v_{0,1}(a,0)=0$, hence $P\in K[x]$. It follows easily that
$1=\deg(P)$.

It remains to consider the case $(1,1)<(\rho,\sigma)<(0,1)$, which implies $\sigma>\rho>0$.
By Theorem~\ref{central} there exist
a $(\rho,\sigma)$-homogenous element $F\in K[x,y]$ such that
$$
[F,\ell_{\rho,\sigma}(P)]=\ell_{\rho,\sigma}(P)\quad\text{and}\quad v_{\rho,\sigma}(F)=\rho+\sigma.
$$
For all $(i,j)\in \Supp(F)$ we have $\rho i+\sigma j=\rho+\sigma$ and
so
$$
(1-i)\rho=(j-1)\sigma.
$$
Hence $j>1$ is impossible and if $j=1$, then $i=1$. Since $\ell_{\rho,\sigma}(P)$ is not a monomial,
we know from Remark~\ref{F no es monomio} that $F$ has at least two points in its support,
hence $(1,1)\in \Supp(F)$ and there must be a point of the form $(i,0) \in \Supp(F)$.
But then $\sigma=(i-1)\rho$ and we obtain $\rho=1$, since $\rho$  and $\sigma$ are coprime.
We obtain that $F=\mu x(y+\lambda x^{\sigma})$ for some $\mu,\lambda\in K^{\times}$
and $\ell_{\rho,\sigma}(P)=x^ap(z)$ for some $p(z)\in K[z]$
with $N:=\deg(p(z))>0$, where $z:=y x^{-\sigma}$.

Consider now the elementary  automorphism $\varphi$ given by
$\varphi(x):=x$ and $\varphi(y):=y-\lambda x^{\sigma}$.
Since  $\varphi$ is $(\rho,\sigma)$-homogenous we have
$\varphi(\ell_{\rho,\sigma}(P))=\ell_{\rho,\sigma}(\varphi(P))$ and so, by
Proposition~\ref{varphi preserva el Jacobiano},
we obtain
$$
[\varphi(F),\ell_{\rho,\sigma}(\varphi(P))]=[\varphi(F),\varphi(\ell_{\rho,\sigma}(P))]
=\varphi(\ell_{\rho,\sigma}(P))=\ell_{\rho,\sigma}(\varphi(P)).
$$

Moreover $\varphi(F)=\mu x y$ is a monomial, hence $\ell_{\rho,\sigma}(\varphi(P))$ is also a monomial.
 It follows that $\ell_{\rho,\sigma}(\varphi(P))=\mu_p x^a z^N$, with $N=\deg(p(z))>0$, since
$\ell_{\rho,\sigma}(\varphi(P))=\varphi(x^ap(z))=x^a p(z-\lambda)$ (Note that $\varphi(z)=z-\lambda$).
 Therefore we arrive at $(a,0)\not\in\Supp(\varphi(P))$.

Now, for $(i,j)\in \Supp(\varphi(P))$ we have
$$
v_{1,1}(i,j)=i+j\le i+\sigma j=v_{\rho,\sigma}(i,j)\le v_{\rho,\sigma}(\varphi(P))=v_{\rho,\sigma}(P)
=v_{\rho,\sigma}(a,0)=a=v_{1,1}(P),
$$
and the equality would be possible only if $j=0$ and $i=a$, but we have $(i,j)\ne(a,0)$. Hence $v_{1,1}(\varphi(P))<v_{1,1}(P)$, as desired.
\end{proof}

\end{document}